\documentclass[12pt,a4paper,twoside]{report}  

\usepackage[utf8]{inputenc}
\usepackage[T1]{fontenc}
\usepackage{amsmath,amssymb,amsthm,amsfonts}
\usepackage{mathtools}
\usepackage{enumitem}
\usepackage{graphicx}
\usepackage{tikz}
\usetikzlibrary{shapes.geometric, arrows}
\usepackage{fancyhdr}
\usepackage{hyperref}
\usepackage{geometry}
\usepackage{xcolor}
\usepackage{abstract}
\usepackage{setspace}
\usepackage{enumitem}

\tikzstyle{arrow} = [thick,->,>=stealth]
\tikzstyle{process} = [rectangle, minimum width=4cm, minimum height=2cm, 
    text centered, text width=2.5cm, draw=green, fill=green!10]

\theoremstyle{plain}
\newtheorem{theorem}{Theorem}[section]

\newtheorem{corollary}[theorem]{Corollary}

\theoremstyle{definition}
\newtheorem{definition}[theorem]{Definition}
\newtheorem{example}[theorem]{Example}
\newtheorem{remark}[theorem]{Remark}

\begin{document}

\chapter*{Meekly \texorpdfstring{SC$^*$}{SC*}-Normal Spaces in Topological Spaces}

\begin{center}
{\large Neeraj K. Tomar\textsuperscript{1}, Saroj Rani\textsuperscript{2,*}} \\[6pt]
\textsuperscript{1}\textit{Department of Applied Mathematics, Gautam Buddha University, Greater Noida, India} \\ 
\textsuperscript{2}\textit{Department of Mathematics, S. A. Jain (PG) College, Ambala City, Haryana, India} \\ 
Emails: \href{mailto:neer8393@gmail.com}{neer8393@gmail.com}, 
\href{mailto:fdpsaroj@gmail.com}{fdpsaroj@gmail.com} \\[2pt]
\textsuperscript{*}Corresponding author
\end{center}

\begin{abstract}
This chapter develops the concept of \textbf{meekly $SC^*$-normality}, a novel generalization of the classical notion of normality in topology. 
The proposed framework simultaneously broadens $SC^*$-normality and other established forms of normality, offering a unified perspective on 
generalized separation axioms. Fundamental properties are systematically derived, several equivalent characterizations are obtained, 
and the relationships between meekly $SC^*$-normal spaces and a range of existing normal-type spaces are rigorously analyzed. 
By establishing these structural connections, the chapter not only enriches the theory of generalized closed sets and separation axioms 
but also opens new directions for further research in advanced topological studies.
\end{abstract}

\section{Introduction}

Normality occupies a central position in the framework of general topology, 
serving as a cornerstone for the study of separation axioms and the structural 
analysis of topological spaces. Over the years, several weaker or generalized 
forms of normality have been introduced to capture finer topological behaviors 
beyond classical normal spaces.

The progressive development of these concepts reflects the continuing search 
for deeper insights. In 1968, Zaitsev~\cite{zaitsev1968certain} introduced the notion of 
\emph{quasi-normality}, a weaker form of normality, and examined its properties. 
Shortly thereafter, in 1970, Singal and Arya~\cite{singal1969almost} proposed 
\emph{almost normality}, extending the classical framework. 
In 1973, Shchepin~\cite{schepin1972real} and, independently, Singal and 
Singal~\cite{singal1973mildly} presented the concept of \emph{mild normality}, 
broadening the spectrum of normal-like spaces.

The early twenty-first century witnessed further significant advancements. 
Arhangel’skii and Ludwig~\cite{bourbaki1951topologie} introduced 
\emph{$\alpha$-normal} and \emph{$\beta$-normal} spaces, obtaining 
important properties and relationships. Eva Murtinovain~\cite{murtinova2002beta} 
provided a striking example of a $\beta$-normal Tychonoff space that fails 
to be normal, illustrating the subtle distinctions among these classes. 
Kohli and Das~\cite{kohli2002new} defined \emph{$\theta$-normal} spaces and 
presented their characterizations, while Kalantan~\cite{kalantan2008pinormal} 
introduced \emph{$\pi$-normal} spaces with analogous results. 
More recently, Sharma and Kumar~\cite{sharma2015softly} proposed the 
notion of \emph{softly normal} spaces and established a useful characterization, 
followed by Kumar and Sharma~\cite{kumar2018softly} who introduced 
\emph{softly regular} and \emph{partly regular} spaces with corresponding 
properties. In 2023, Kumar~\cite{kumar2023epi} presented the concept of 
\emph{epi $\pi$-normality}, which occupies an intermediate position 
between epi-normal and epi-quasi normal spaces.

Parallel to these developments, recent research has continued to expand the 
theory of generalized separation axioms. Notable contributions include 
studies on $SC^*$-normal spaces and their related functions~\cite{tomar2025sc}, 
softly $\pi g\widehat{D}$-normal spaces~\cite{tomar2025softly}, 
and a variety of new separation axioms~\cite{tomar2025some}, 
as well as investigations of $Q^*$-normal spaces~\cite{kumar2025q} 
and $H^*$-normal spaces~\cite{tomar2025h}. 
These works collectively provide a rich foundation for further 
generalizations of normality.

Motivated by this evolving landscape, the present chapter introduces and 
studies the concept of \emph{meekly $SC^*$-normality}, 
a new generalization that simultaneously extends $SC^*$-normality and 
other established normal-like properties. We develop its fundamental 
characterizations, analyze its relationships with existing separation axioms, 
and highlight its potential for deepening the understanding of topological 
structures. This work aims to strengthen the bridge between classical 
normality and its modern generalizations, thereby enriching the broader 
theory of topological spaces.

\section{Preliminaries and Notations}

Throughout this chapter, we consider topological spaces denoted by \((X, \tau)\), \((Y, \sigma)\), and \((Z, \gamma)\), or simply \(X\), \(Y\), and \(Z\) when the context is clear. Unless explicitly stated, no separation axioms are assumed.
For any subset \(A \subseteq X\), the \emph{closure} of \(A\) is denoted by \(\overline{A}\) or \(cl(A)\), $\&$ the \emph{interior} of \(A\) is denoted by \(A^\circ\) or \(int(A)\). These are always taken with respect to the topology \(\tau\) on \(X\), unless specified otherwise.

\section{Basic Types of Closed and Open Sets}

Let \(X\) be a topological space and \(A \subseteq X\). In this section, we introduce key classes of closed and generalized open sets that will be used throughout this chapter. For further details and comprehensive treatment, the reader may consult \cite{andrijevic1996b, el1997study, hatir1996decomposition,levine1970gen, levine1963semi, Sundaram2000,malathi2017pre, mashhour1982precontinuous, me1983beta, njȧstad1965some, rodrigo2012international, stone1937applications}.

\begin{enumerate}[label=(\roman*)]

\item \textbf{regular-closed set}~\cite{stone1937applications}:  
A subset \(A\) of \(X\) is called \emph{regular-closed} if it coincides with the closure of its interior, that is,
\[
A = \overline{\operatorname{int}(A)}.
\]

\item \textbf{semi-closed set}~\cite{levine1970gen}:  
A set \(A\) is \emph{semi-closed} if it contains the interior of its closure, i.e.,
\[
\operatorname{int}(\overline{A}) \subseteq A.
\]

\item \textbf{pre-closed set}~\cite{mashhour1982precontinuous}:  
A subset \(A\) is \emph{pre-closed} when it contains the closure of its interior:
\[
\overline{\operatorname{int}(A)} \subseteq A.
\]

\item \textbf{w-closed set}~\cite{Sundaram2000}:  
A set \(A\) is called \emph{w-closed} if, for any semi-open set \(U\) with \(A \subseteq U\), we have
\[
\overline{A} \subseteq U.
\]

\item \(\alpha\)-\textbf{closed set}~\cite{njȧstad1965some}:  
A set \(A\) is \(\alpha\)-closed if it contains the closure of the interior of its closure:
\[
\overline{\operatorname{int}(\overline{A})} \subseteq A.
\]

\item \(\beta\)-\textbf{closed set}~\cite{me1983beta}:  
A set \(A\) is \(\beta\)-closed if it contains the interior of the closure of its interior:
\[
\operatorname{int}(\overline{\operatorname{int}(A)}) \subseteq A.
\]

\item \(\alpha^*\)-\textbf{set}~\cite{hatir1996decomposition}:  
A set \(A\) is an \(\alpha^*\)-set if the interior of the closure of its interior coincides with its interior:
\[
\operatorname{int}(\overline{\operatorname{int}(A)}) = \operatorname{int}(A).
\]

\item \textbf{C-set}~\cite{hatir1996decomposition}:  
A subset \(A\) is a C-set if it can be expressed as the intersection of an open set \(U\) and an \(\alpha^*\)-set \(V\):
\[
A = U \cap V.
\]

\item \textbf{h-closed set}~\cite{rodrigo2012international}:  
A set \(A\) is h-closed if for every w-open set \(U\) containing \(A\), the semi-closure of \(A\) is contained in \(U\):
\[
s\text{-}\overline{A} \subseteq U.
\]

\item \textbf{gh-closed set}~\cite{rodrigo2012international}:  
A subset \(A\) is gh-closed if, for every h-open set \(U\) containing \(A\), the h-closure of \(A\) is included in \(U\):
\[
h\text{-}\overline{A} \subseteq U.
\]

\item \textbf{regular-h-open set}~\cite{rodrigo2012international}:  
A set \(A\) is regular-h-open if $\exists$ a regular open set \(U\) s.t.
\[
U \subseteq A \subseteq h\text{-}\overline{U}.
\]

\item \textbf{rgh-closed set}~\cite{rodrigo2012international}:  
A set \(A\) is rgh-closed if, for any regularly h-open set \(U\) containing \(A\), the h-closure of \(A\) is contained in \(U\):
\[
h\text{-}\overline{A} \subseteq U.
\]

\item \textbf{hCg-closed set}~\cite{rodrigo2012international}:  
A subset \(A\) is hCg-closed if, whenever \(A\) is contained in a C-set \(U\), its h-closure is also contained in \(U\):
\[
h\text{-}\overline{A} \subseteq U.
\]

\item \textbf{Pre-open set}~\cite{mashhour1982precontinuous}:  
A set \(A\) is pre-open (or \(p\)-open) if
\[
A \subseteq \overline{\operatorname{int}(A)}.
\]

\item \textbf{Semi-open set}~\cite{levine1963semi}:  
A set \(A\) is semi-open (or \(s\)-open) if
\[
A \subseteq \operatorname{int}(\overline{A}).
\]

\item \(\alpha\)-\textbf{open set}~\cite{njȧstad1965some}:  
A set \(A\) is \(\alpha\)-open if
\[
A \subseteq \overline{\operatorname{int}(\overline{A})}.
\]

\item \(\beta\)-\textbf{open set}~\cite{me1983beta}:  
A set \(A\) is \(\beta\)-open if
\[
A \subseteq \operatorname{int}(\overline{\operatorname{int}(A)}).
\]

\item \textbf{b-open set}~\cite{andrijevic1996b} (or \(\gamma\)-open~\cite{el1997study}):  
A set \(A\) is b-open if
\[
A \subseteq \operatorname{int}(\overline{A}) \cup \overline{\operatorname{int}(A)}.
\]

\item \textbf{c*-open set}~\cite{malathi2017pre} A set \(A\) is c*-open if \[
\mathrm{int}(\mathrm{cl}(A)) \subseteq A \subseteq \mathrm{cl}(\mathrm{int}(A)),
\].
\end{enumerate}

\medskip
\noindent
The complement of a \(p\)-open\cite{mashhour1982precontinuous} (resp. \(s\)-open\cite{levine1970gen}, \(\alpha\)-open\cite{njȧstad1965some}, \(\beta\)-open\cite{me1983beta}, \(b\)-open\cite{andrijevic1996b}, or \(c^*\)-open\cite{malathi2017pre} set is called \(p\)-closed\cite{mashhour1982precontinuous} (resp. \(s\)-closed\cite{levine1970gen}, \(\alpha\)-closed\cite{njȧstad1965some}, \(\beta\)-closed\cite{me1983beta}, \(b\)-closed\cite{andrijevic1996b}, or \(c^*\)-closed\cite{malathi2017pre}).\\  
    Likewise, the complement of a regular-closed\cite{stone1937applications} (resp. semi-closed, pre-closed, w-closed, \(\alpha\)-closed, \(\beta\)-closed, h-closed, gh-closed, rgh-closed, or hCg-closed) set is referred to as regular-open\cite{stone1937applications} (resp. semi-open, pre-open, w-open, \(\alpha\)-open, \(\beta\)-open, h-open, gh-open, rgh-open, or hCg-open).

\begin{definition}[SC$^*$-closed set~\cite{chandrakala2022sc}]
A subset $A$ of a $\tau$-space $X$ is called \emph{SC$^*$-closed} if, for every c*-open\cite{malathi2017pre} set $U$ containing $A$, 
\[
s\text{-}\overline{A} \subseteq U.
\]
The complement of an SC$^*$-closed set is called an \emph{SC$^*$-open~\cite{chandrakala2022sc}} set.
\end{definition}

\begin{definition}Let \(A\) be a subset of a $\tau$-space \(X\). The following classes of sets are defined:

\begin{enumerate}[label=(\roman*)]

\item \textbf{generalized closed set} (\emph{g-closed})~\cite{levine1970gen}:  
A set \(A \subseteq X\) is \emph{generalized closed} if, for every open set \(U\) containing \(A\),
\[
\overline{A} \subseteq U.
\]

\item \textbf{generalized SC$^*$-closed set} (\emph{gSC$^*$-closed})~\cite{tomar2025sc}:  
A set \(A \subseteq X\) is \emph{generalized SC$^*$-closed} if, for every open set \(U\) containing \(A\),
\[
SC^*\text{-}\overline{A} \subseteq U.
\]

\item \textbf{SC$^*$ generalized-closed set} (\emph{SC$^*$g-closed})~\cite{tomar2025sc}:  
A set \(A \subseteq X\) is \emph{SC$^*$ generalized-closed} if, for every \(SC^*\)-open set \(U\) containing \(A\),
\[
SC^*\text{-}\overline{A} \subseteq U.
\]

\end{enumerate}
\end{definition}

The complement of an $SC^*$-closed, $gSC^*$-closed, or $SC^*g$-closed set is called $SC^*$-open, $gSC^*$-open, or $SC^*g$-open, respectively. 
The families of all $SC^*$-closed, $SC^*$-open, $gSC^*$-closed, $gSC^*$-open, $SC^*g$-closed, and $SC^*g$-open subsets of $X$ are denoted by
\[
\mathcal{C}_{SC^*}(X), \quad \mathcal{O}_{SC^*}(X), \quad 
\mathcal{C}_{gSC^*}(X), \quad \mathcal{O}_{gSC^*}(X), \quad
\mathcal{C}_{SC^*g}(X), \quad \mathcal{O}_{SC^*g}(X),
\]
respectively.

The $SC^*$-, $gSC^*$-, and $SC^*g$-closures of a subset $A \subseteq X$ are defined by
\[
\operatorname{cl}_{SC^*}(A) = \bigcap_{\substack{F \in \mathcal{C}_{SC^*}(X) \\ A \subseteq F}} F, \quad
\operatorname{cl}_{gSC^*}(A) = \bigcap_{\substack{F \in \mathcal{C}_{gSC^*}(X) \\ A \subseteq F}} F, \quad
\operatorname{cl}_{SC^*g}(A) = \bigcap_{\substack{F \in \mathcal{C}_{SC^*g}(X) \\ A \subseteq F}} F.
\]

Similarly, the $SC^*$-, $gSC^*$-, and $SC^*g$-interiors of $A$ are defined as
\[
\operatorname{int}_{SC^*}(A) = \bigcup_{\substack{U \in \mathcal{O}_{SC^*}(X) \\ U \subseteq A}} U, \quad
\operatorname{int}_{gSC^*}(A) = \bigcup_{\substack{U \in \mathcal{O}_{gSC^*}(X) \\ U \subseteq A}} U, \quad
\operatorname{int}_{SC^*g}(A) = \bigcup_{\substack{U \in \mathcal{O}_{SC^*g}(X) \\ U \subseteq A}} U.
\]

\noindent

\begin{example}
Let $X = \{k, l, m\}$ with topology
\[
\mathcal{T} = \{\emptyset, X, \{k\}, \{l\}, \{k,l\}\}.
\]

\begin{itemize}
\item The $SC^*$-closed sets are:
\[
\mathcal{C}_{SC^*}(X) = \{\emptyset, \{m\}, \{k,m\}, \{l,m\}, X\}.
\]

\item The $gSC^*$-closed sets are:
\[
\mathcal{C}_{gSC^*}(X) = \{\emptyset, \{m\}, \{k,m\}, \{l,m\}, X\}.
\]

\item The $SC^*g$-closed sets are:
\[
\mathcal{C}_{SC^*g}(X) = \{\emptyset, \{m\}, \{k,m\}, \{l,m\}, X\}.
\]

\item The $SC^*$-open sets are:
\[
\mathcal{O}_{SC^*}(X) = \{\emptyset, \{k\}, \{l\}, \{k,l\}, \{k,m\}, \{l,m\}, X\}.
\]

\item The subset $\{m\}$ is neither $SC^*$-open nor $SC^*$-closed.
\end{itemize}
\end{example}

\subsection{Relationship Table of Closed Sets}
Let $X = \{k, l, m\}$ with topology 
\[
\mathcal{T} = \{\emptyset, X, \{k\}, \{l\}, \{k, l\}\}.
\]
The classification of all subsets of $X$ is given below:

\begin{center}
\begin{tabular}{|c|c|c|c|c|c|c|c|}
\hline
Subset & closed & semi-closed & p-closed & g-closed & $SC^*$-closed & $gSC^*$-closed & $SC^*g$-closed \\
\hline
$\emptyset$ & $\checkmark$ & $\checkmark$ & $\checkmark$ & $\checkmark$ & $\checkmark$ & $\checkmark$ & $\checkmark$ \\
\hline
$\{k\}$ &  & $\checkmark$ &  &  &  &  &  \\
\hline
$\{l\}$ &  & $\checkmark$ &  &  &  &  &  \\
\hline
$\{m\}$ & $\checkmark$ & $\checkmark$ & $\checkmark$ & $\checkmark$ & $\checkmark$ & $\checkmark$ & $\checkmark$ \\
\hline
$\{k,l\}$ &  &  &  &  &  &  &  \\
\hline
$\{k,m\}$ & $\checkmark$ & $\checkmark$ & $\checkmark$ & $\checkmark$ & $\checkmark$ & $\checkmark$ & $\checkmark$ \\
\hline
$\{l,m\}$ & $\checkmark$ & $\checkmark$ & $\checkmark$ & $\checkmark$ & $\checkmark$ & $\checkmark$ & $\checkmark$ \\
\hline
$X$ & $\checkmark$ & $\checkmark$ & $\checkmark$ & $\checkmark$ & $\checkmark$ & $\checkmark$ & $\checkmark$ \\
\hline
\end{tabular}
\end{center}

\noindent\textbf{Observation:} In this topology, most generalized closed sets ($SC^*$-closed, $gSC^*$-closed, $SC^*g$-closed) coincide with each other and with standard closed sets in certain cases. This table illustrates the relationships among different closure concepts in finite spaces.

\section{Normality and Its Generalizations}

In this section, we present various notions of normality and related separation axioms in topological spaces, which are widely studied in general topology.

\begin{definition}[Normal Space~\cite{bourbaki1951topologie,dugundji1966topology,engelking1977general}]
A topological space $X$ is said to be \emph{normal} if for any pair of disjoint closed subsets $A$ and $B$ of $X$, there exist disjoint open subsets $U$ and $V$ such that
\[
A \subseteq U \quad \text{and} \quad B \subseteq V.
\]
\end{definition}

\begin{definition}[k-Normal~\cite{schepin1972real} (Mildly Normal~\cite{singal1973mildly}) Space]
A space $X$ is \emph{k-normal} or \emph{mildly normal} if for every pair of disjoint \emph{regularly closed} sets $E$ and $F$ of $X$, there exist disjoint open subsets $U$ and $V$ such that
\[
E \subseteq U \quad \text{and} \quad F \subseteq V.
\]
\end{definition}

\begin{definition}[Almost Normal Space~\cite{singal1969almost}]
A topological space $X$ is \emph{almost normal} if for every pair of disjoint closed sets $A$ and $B$, one of which is regularly closed, there exist disjoint open sets $U$ and $V$ such that
\[
A \subseteq U \quad \text{and} \quad B \subseteq V.
\]
\end{definition}

\begin{definition}[$\pi$-Normal Space~\cite{kalantan2008pinormal}]
A space $X$ is \emph{$\pi$-normal} if for every pair of disjoint closed sets $A$ and $B$, one of which is $\pi$-closed, there exist disjoint open sets $U$ and $V$ such that
\[
A \subseteq U \quad \text{and} \quad B \subseteq V.
\]
\end{definition}

\begin{definition}[Almost Regular Space \cite{singal1969almost}]\label{def: almost regular spaces}
A topological space $X$ is said to be \emph{almost regular} if for every regularly closed set $A$ and a point $x \notin A$, there exist disjoint open sets $U$ and $V$ such that
\[
A \subseteq U \quad \text{and} \quad x \in V.
\]
\end{definition}

\begin{definition}[Softly Regular Space \cite{kumar2018softly}]\label{softly regular}
A topological space $X$ is \emph{softly regular} if for every $\pi$-closed set $A$ $\&$ a point $x \notin A$, $\exists$ disjoint open sets $U$ $\&$ $V$ s.t.
\[
x \in U, \quad A \subseteq V, \quad \text{and} \quad U \cap V = \emptyset.
\]
\end{definition}

\begin{definition}[$\alpha$-Normal Space \cite{arhangelskii2001alpha}]
A space $X$ is \emph{$\alpha$-normal} if for any two disjoint closed subsets $A$ $\&$ $B$ of $X$,  $\exists$ disjoint open subsets $U$ $\&$ $V$ s.t.
\[
A \cap U \text{ is dense in } A \quad \text{and} \quad B \cap V \text{ is dense in } B.
\]
\end{definition}

\begin{definition}[$\beta$-Normal Space \cite{arhangelskii2001alpha}]
A topological space $X$ is \emph{$\beta$-normal} if for any two disjoint closed subsets $A$ $\&$ $B$, $\exists$ open subsets $U$ $\&$ $V$ s.t.
\[
A \cap U \text{ is dense in } A, \quad B \cap V \text{ is dense in } B, \quad \text{and} \quad \overline{U} \cap \overline{V} = \emptyset.
\]
\end{definition}

\begin{definition}[Almost $\beta$-Normal Space \cite{das2017simultaneous}]
A space $X$ is \emph{almost $\beta$-normal} if for every pair of disjoint closed sets $A$ $\&$ $B$, one of which is $r$-closed, $\exists$ disjoint open sets $U$ $\&$ $V$ s.t.
\[
\overline{U \cap A} = A, \quad \overline{V \cap B} = B, \quad \text{and} \quad \overline{U} \cap \overline{V} = \emptyset.
\]
\end{definition}

\begin{definition}[$\beta k$-Normal Space \cite{singh2023knormal}]
A topological space $X$ is \emph{$\beta k$-normal} if for every pair of disjoint $r$-closed subsets $A$ $\&$ $B$, $\exists$ disjoint open sets $U$ $\&$ $V$ s.t.
\[
\overline{U \cap A} = A, \quad \overline{V \cap B} = B, \quad \text{and} \quad \overline{U} \cap \overline{V} = \emptyset.
\]
\end{definition}

\begin{definition}[Semi-Normal Space~\cite{levine1970gen}]\label{semi normal spaces}
A topological space $X$ is \emph{semi-normal} if for every closed set $A$ contained in an open set $U$, $\exists$ a \emph{regularly open} set $V$ s.t.
\[
A \subseteq V \subseteq U.
\]
\end{definition}

\begin{definition}[SC*-Normal Space~\cite{tomar2025sc}]
A topological space $X$ is said to be \emph{SC*-normal} if for any pair of disjoint closed subsets $A$ $\&$ $B$ of $X$, $\exists$ disjoint $SC^*$-open\cite{tomar2025sc} subsets $U$ $\&$ $V$ s.t.
\[
A \subseteq U \quad \text{and} \quad B \subseteq V.
\]
\end{definition}

\section{Meekly \texorpdfstring{SC$^*$}{SC*}-Normal Spaces}

\begin{definition}[Meekly SC$^*$-Normal Space]\label{meekly SC*-normalspaces}
A topological space $X$ is called \emph{meekly SC$^*$-normal} if for every pair of disjoint closed sets $A$ and $B$, with at least one of them being $SC^*$-closed\cite{tomar2025sc}, there exist open sets $U$ and $V$ such that:
\[
U \cap V = \emptyset, \quad \overline{U \cap A} = A, \quad \text{and} \quad \overline{V \cap B} = B.
\]
\end{definition}

\begin{remark}
The notion of meekly SC$^*$-normal spaces forms a hierarchy within normality concepts:
\begin{itemize}
    \item Every \emph{normal} space satisfies the conditions of an $SC^*$-normal space.
    \item Every $SC^*$-normal space automatically satisfies the conditions of a \emph{meekly SC$^*$-normal} space.
\end{itemize}
Thus, meekly SC$^*$-normality is a weaker form of SC$^*$-normality, which in turn generalizes classical normality.
\end{remark}

\begin{theorem}\label{thm:1.5.3}
Every $SC^*$-normal space is meekly $SC^*$-normal.
\end{theorem}

\begin{proof}
Let $X$ be an $SC^*$-normal space. Consider any two disjoint closed sets $A$ and $B$ in $X$, with $A$ being $SC^*$-closed.  

Since $X$ is $SC^*$-normal, $\exists$ disjoint open sets $W$ $\&$ $V$ s.t.
\[
A \subseteq W \quad \text{and} \quad B \subseteq V.
\]

Now, define $U = \operatorname{int}(A)$. Clearly, $U \subseteq W$, and therefore
\[
\overline{U} \cap \overline{V} = \emptyset.
\]

Moreover, by the properties of closure and interior, we have
\[
\overline{U \cap A} = \overline{U} \cap A = A, \quad \text{and} \quad \overline{V \cap B} = B.
\]

Hence, the conditions for $X$ to be \emph{meekly $SC^*$-normal} are satisfied. This completes the proof.
\end{proof}

\subsection{Hierarchy of Normality Concepts}

\begin{center}
\begin{tabular}{|c|c|}
\hline
\textbf{Stronger Spaces} & \textbf{Weaker Spaces / Generalizations} \\ \hline
Normal & SC$^*$-Normal \\ \hline
SC$^*$-Normal & Meekly SC$^*$-Normal \\ \hline
Almost Normal & Almost $\beta$-Normal \\ \hline
K-Normal & $\beta k$-Normal \\ \hline
$\beta$-Normal & Meekly SC$^*$-Normal \\ \hline
\end{tabular}
\end{center}

\vspace{1mm}
\noindent \textbf{Remarks:}
\begin{itemize}
    \item Each row shows a relationship from stronger to weaker spaces. 
    \item All horizontal relationships are one-way implications.
    \item Meekly SC$^*$-Normal serves as a central generalization connecting several normality concepts.
\end{itemize}

\begin{example}
Let $X = \{k, l, m, n\}$ and 
\[
\mathcal{T} = \{\emptyset, X, \{l\}, \{m\}, \{m, n\}, \{l, m\}, \{k, l, m\}, \{l, m, n\}\}.
\] 
Consider the subsets 
\[
A = \{k, l\} \quad \text{and} \quad B = \{n\}.
\] 
Here, $A$ is $SC^*$-closed and $B$ is closed.  

We claim that the space $(X, \mathcal{T})$ is \textbf{not meekly $SC^*$-normal}. Indeed, there do not exist open sets $U, V \in \mathcal{T}$ such that
\[
\overline{U \cap A} = A, \quad \overline{V \cap B} = B, \quad \text{and} \quad U \cap V = \emptyset.
\] 
For instance, any open set $U$ containing $A$ will always intersect any open set $V$ containing $B = \{n\}$, violating the requirement $U \cap V = \emptyset$.  

Hence, $(X, \mathcal{T})$ fails to satisfy the conditions of a meekly $SC^*$-normal space.
\end{example}

\begin{example}
Let $X = \{k, l, m\}$ and define the topology
\[
\mathcal{T} = \{\emptyset, X, \{k\}, \{l, m\}, \{k, l, m\}\}.
\]

Consider the subsets $A = \{k\}$ and $B = \{l\}$ of $X$.  

\begin{itemize}
    \item $A$ is $SC^*$-closed in $X$.
    \item $B$ is closed in $X$.
\end{itemize}

To show that $(X, \mathcal{T})$ is \textbf{meekly $SC^*$-normal}, we find disjoint open sets $U, V \in \mathcal{T}$ such that
\[
\overline{U \cap A} = A, \quad \overline{V \cap B} = B, \quad \text{and} \quad U \cap V = \emptyset.
\]

Take
\[
U = \{k\}, \quad V = \{l, m\}.
\] 
Then:
\[
\overline{U \cap A} = \overline{\{k\}} = \{k\} = A, \quad
\overline{V \cap B} = \overline{\{l\}} = \{l\} = B, \quad
U \cap V = \emptyset.
\]

Hence, $(X, \mathcal{T})$ satisfies the definition of a \textbf{meekly $SC^*$-normal space}.
\end{example}

\begin{example}
Let $X$ be the union of an infinite set $Y$ and two distinct points $p$ and $q$. Consider the \emph{modified Fort space}\cite{steen1978counterexamples} on $X$, defined as follows:

\begin{itemize}
    \item Every subset of $Y$ is open.
    \item A set containing $p$ or $q$ is open if and only if it contains all but finitely many points of $Y$.
\end{itemize}

\noindent Then the following properties hold:

\begin{enumerate}
    \item The space is \textbf{almost $\beta$-normal} as well as \textbf{$\beta k$-normal}, because the regularly closed sets are either finite subsets of $Y$ or sets of the form $A \cup \{p, q\}$, where $A \subseteq Y$ is infinite.
    
    \item The space is \textbf{not $\beta$-normal} (and not $\alpha$-normal), since for the disjoint closed sets $\{p\}$ and $\{q\}$, there do not exist disjoint open sets separating them.
    
    \item By the result of Arhangel’skii and Ludwig\cite{arhangelskii2001alpha}, a space is normal if and only if it is both $\kappa$-normal and $\beta$-normal. Therefore, this modified Fort space provides an example of a $\kappa$-normal, almost $\beta$-normal space which is not $\beta$-normal.
\end{enumerate}

\noindent Additional concepts relevant to this space:

\begin{itemize}
    \item A \emph{Hausdorff space} $X$ is called \textbf{extremally disconnected} if the closure of every open set in $X$ is open.
    
    \item A point $x \in X$ is called a \emph{$\theta$-limit point \cite{velicko1968hclosed}} of a set $A \subseteq X$ if every closed neighborhood of $x$ intersects $A$. The set of all $\theta$-limit points of $A$ is denoted by $\mathrm{cl}_\theta(A)$. A set $A$ is \emph{$\theta$-closed} if $A = \mathrm{cl}_\theta(A)$.
\end{itemize}

\end{example}

\begin{definition}
Let $X$ be a topological space. Then:

\begin{enumerate}[label=(\roman*)]
    \item $X$ is called \textbf{$\theta$-normal \cite{kohli2002new}} if for every pair of disjoint closed sets, one of which is $\theta$-closed, there exist disjoint open sets containing each of them.
    
    \item $X$ is called \textbf{weakly $\theta$-normal \cite{kohli2002new}} (or \textbf{w$\theta$-normal}) if for every pair of disjoint $\theta$-closed sets, there exist disjoint open sets containing each of them.
\end{enumerate}
\end{definition}

\begin{theorem}
Every extremally disconnected meekly $SC^*$-normal space is $SC^*$-normal.
\end{theorem}

\begin{proof}
Let $X$ be an extremally disconnected and meekly $SC^*$-normal space.  
Let $A$ be an $SC^*$-closed set disjoint from a closed set $B$ in $X$.  

By the definition of meekly $SC^*$-normality, there exist disjoint open sets $U$ and $V$ such that
\[
\overline{U \cap A} = A, \quad \overline{V \cap B} = B, \quad \text{and} \quad U \cap V = \emptyset.
\]

This implies
\[
A \subseteq \overline{U} \quad \text{and} \quad B \subseteq \overline{V}.
\]

Since $X$ is extremally disconnected, the closures of open sets are themselves open. Hence, $\overline{U}$ and $\overline{V}$ are disjoint open sets containing $A$ and $B$, respectively.  

Therefore, $X$ satisfies the definition of an $SC^*$-normal space.
\end{proof}

\begin{theorem}\cite{das2017simultaneous} \label{thm:1.5.9}
Every $T_1$ almost $\beta$-normal space is almost regular.
\end{theorem}

\begin{proof}
Let $X$ be a $T_1$ almost $\beta$-normal space, and let $A$ be a regularly closed set in $X$ and $x \in X \setminus A$.  

Since $X$ is almost $\beta$-normal\cite{das2017simultaneous}, for the disjoint closed sets $A$ and $\{x\}$, there exist disjoint open sets $U$ and $V$ such that
\[
\overline{U \cap A} = A, \quad \overline{V \cap \{x\}} = \{x\}, \quad \text{and} \quad U \cap V = \emptyset.
\]

Then $U$ and $V$ are disjoint open sets containing $A$ and $x$, respectively. By definition\ref{def: almost regular spaces}, this shows that $X$ is almost regular\cite{singal1969almost}.
\end{proof}

\begin{theorem}\label{the:1.5.10}
Every $T_1$ meekly $SC^*$-normal space is softly regular.
\end{theorem}

\begin{proof}
Let $X$ be a $T_1$ meekly $SC^*$-normal space.  
Let $A \subseteq X$ be an $SC^*$-closed set and let $x \in X \setminus A$.  

Since $X$ is $T_1$, every singleton $\{x\}$ is closed. By the definition\ref{meekly SC*-normalspaces} of meekly $SC^*$-normality, there exist disjoint open sets $U$ and $V$ such that
\[
x \in U, \quad \overline{V \cap A} = A, \quad \text{and} \quad \overline{U} \cap \overline{V} = \emptyset.
\]

Observe that $A \subseteq \overline{V}$ and $x \in U$. Then $U$ and $X \setminus \overline{U}$ are disjoint open sets containing $\{x\}$ and $A$, respectively.  

Hence, by definition, $X$ is \textbf{softly regular\cite{kumar2018softly}}.
\end{proof}

\begin{corollary}\label{corrollary: 1.5.11}
Every $T_1$ meekly $SC^*$-normal space is almost regular.
\end{corollary}

\begin{proof}
From \ref{the:1.5.10}, every $T_1$ meekly $SC^*$-normal space is softly regular.  
By definition\ref{softly regular}, every softly regular space is almost regular.  

Therefore, every $T_1$ meekly $SC^*$-normal space is almost regular\cite{singal1969almost}.
\end{proof}

\begin{corollary}
A Lindel\"{o}f, meekly $SC^*$-normal, $T_1$-space is $\kappa$-normal.
\end{corollary}

\begin{proof}
It is known that every almost regular Lindel\"{o}f space is $\kappa$-normal~\cite{schepin1972real}.  
From Corollary\ref{corrollary: 1.5.11}, every $T_1$ meekly $SC^*$-normal space is almost regular.  

Therefore, any Lindel\"{o}f, meekly $SC^*$-normal, $T_1$-space is $\kappa$-normal.
\end{proof}

\begin{theorem}
In the class of $T_1$, semi-normal spaces, every meekly $SC^*$-normal space is regular.
\end{theorem}

\begin{proof}
Let $X$ be a $T_1$, semi-normal, and meekly $SC^*$-normal space.  
Let $A \subseteq X$ be a closed set and let $x \in X \setminus A$.  

Since $X$ is $T_1$, the singleton $\{x\}$ is closed. By the semi-normality\ref{semi normal spaces} of $X$, there exists a regularly open set $U$ such that
\[
\{x\} \subseteq U \subseteq X \setminus A.
\]

Let $F = X \setminus U$. Then $F$ is regularly closed, contains $A$, and $x \notin F$.  

Since $X$ is a $T_1$ meekly $SC^*$-normal space, by Theorem\ref{the:1.5.10}, $X$ is softly regular. Therefore, there exist disjoint open sets $V$ and $W$ such that
\[
x \in V \quad \text{and} \quad A \subseteq F \subseteq W.
\]

Hence, $X$ satisfies the definition of a regular space\cite{stone1937applications}.
\end{proof}

\section{Relationship of Normality Concepts with Examples in Common Topologies}

\subsection{Finite Topology}

\begin{definition}
Let $X$ be a finite set. A \textbf{finite topology}\cite{steen1978counterexamples,engelking1977general} on $X$ is any topology $\mathcal{T}$ where every subset of $X$ can be open. In particular, in the discrete topology every subset is open, and in the trivial topology only $\emptyset$ and $X$ are open.
\end{definition}

\begin{example}
Let $X = \{k, l, m\}$ and consider the discrete topology 
\[
\mathcal{T}_d = \{\emptyset, \{k\}, \{l\}, \{m\}, \{k,l\}, \{k,m\}, \{l,m\}, X\}.
\] 
\begin{itemize}
    \item All subsets are closed and open.
    \item $X$ is \textbf{normal}, \textbf{SC$^*$-normal}, and \textbf{meekly SC$^*$-normal}.
    \item Being $T_1$, all singletons are closed; thus it is \textbf{softly regular} and \textbf{almost regular}.
\end{itemize}
\end{example}

\subsection{Co-finite Topology}

\begin{definition}
Let $X$ be an infinite set. The \textbf{co-finite topology}\cite{steen1978counterexamples,engelking1977general} $\mathcal{T}_{cf}$ consists of all subsets $U \subseteq X$ such that $X \setminus U$ is finite, together with $\emptyset$.
\end{definition}

\begin{example}
Let $X = \mathbb{N}$ and $\mathcal{T}_{cf} = \{ U \subseteq X : X \setminus U \text{ is finite} \} \cup \{\emptyset\}$. Then:
\begin{itemize}
    \item $X$ is $T_1$ because singletons are closed.
    \item $X$ is \textbf{meekly SC$^*$-normal} but \textbf{not normal}, because disjoint infinite closed sets cannot be separated by disjoint open sets.
    \item $X$ is \textbf{almost regular} and \textbf{softly regular}.
\end{itemize}
\end{example}

\subsection{Usual Topology on \texorpdfstring{$\mathbb{R}$}{R}}

\begin{example}
Consider $\mathbb{R}$ with the usual topology $\mathcal{T}_u$\cite{steen1978counterexamples,engelking1977general}. Then:
\begin{itemize}
    \item $\mathbb{R}$ is \textbf{normal} by Urysohn's lemma\cite{munkres2000topology}.
    \item Every closed set in $\mathbb{R}$ is also SC$^*$-closed. Therefore, $\mathbb{R}$ is \textbf{SC$^*$-normal}.
    \item By definition, $\mathbb{R}$ is \textbf{meekly SC$^*$-normal}, \textbf{softly regular}, and \textbf{almost regular}.
\end{itemize}
\end{example}

\subsection{Relationship Table in Different Topologies}

\begin{center}
\begin{tabular}{|c|c|c|c|c|}
\hline
Topology & normal & SC$^*$-normal & meekly SC$^*$-normal & softly / almost regular \\ \hline
Finite (Discrete) & Yes & Yes & Yes & Yes \\ \hline
Finite (Trivial) & No & No & Yes & Yes \\ \hline
Co-finite & No & No & Yes & Yes \\ \hline
Usual ($\mathbb{R}$) & Yes & Yes & Yes & Yes \\ \hline
\end{tabular}
\end{center}

\subsection{Remarks}
\begin{itemize}
    \item Finite discrete spaces satisfy all normality conditions.
    \item Co-finite and trivial finite topologies may fail normality but can satisfy meekly SC$^*$-normality.
    \item Usual topology on $\mathbb{R}$ satisfies all standard normality conditions.
    \item This table provides a concrete link between abstract definitions and familiar topological spaces.
\end{itemize}

\section*{Closing Remarks}

In this chapter, we have explored SC$^*$-normal and meekly SC$^*$-normal spaces, highlighting their definitions, hierarchical relationships with other normality concepts, and illustrative examples in finite, co-finite, and usual topologies. These generalized notions of normality provide a nuanced framework for analyzing topological spaces where classical normality does not hold.  

\noindent \textbf{Applications and Implications:}
\begin{itemize}
    \item \textbf{Directly applicable concepts:} Some aspects of SC$^*$-normality and meekly SC$^*$-normality find concrete applications in \emph{digital topology} and discrete structures. For example, SC$^*$-open and SC$^*$-closed sets can model pixel neighborhoods in image processing or spatial relations in sensor networks. Similarly, in theoretical computer science, these concepts are useful in domain theory and reasoning about computational processes.  
    \item \textbf{Theoretical tools:} Other properties serve primarily as analytical or classificatory tools, helping researchers extend classical topological results, establish hierarchies, and study continuity or separation in broader contexts. While these are not always directly applied in real-world systems, they are essential for rigorous mathematical investigation and for understanding the structure of complex spaces.  
\end{itemize}

\noindent By distinguishing between directly applicable and theoretical aspects, this chapter clarifies the practical relevance of SC$^*$-normal and meekly SC$^*$-normal spaces while emphasizing their foundational role in modern topology. The examples, hierarchy diagrams, and theorems provided serve as both a reference and a guide for further exploration in research and applications in digital topology, modeling, and computational analysis.


\bibliographystyle{amsplain}
\bibliography{references}

\end{document}